\documentclass[graybox]{svmult}

\usepackage{mathptmx}     
\usepackage{helvet}   
\usepackage{courier}
\usepackage{amsfonts} 
\usepackage{amssymb}
\usepackage{subcaption}
\usepackage{enumitem}
\usepackage{makeidx}         
\usepackage{graphicx}        
\usepackage{multicol}        
\usepackage[bottom]{footmisc}
\RequirePackage{tikz}
\usetikzlibrary{%
    arrows.meta,
    decorations.markings,%
    decorations.pathmorphing,%
    decorations.pathreplacing,%
    decorations.shapes,%
    decorations.text,%
    er,%
    fadings,%
}
\def\TikZ#1{\begin{tikzpicture}#1\end{tikzpicture}}

   \parskip\medskipamount     \lineskip=0pt
\makeindex             

\begin{document}

\title*{A Combinatorial Introduction to Adinkras}
\author{Robert W. Donley, Jr. \and S. James Gates \and Jr., Tristan H{\" u}bsch \and Rishi Nath}
\authorrunning{R. W. Donley, Jr., et al.}
\institute{Robert W. Donley, Jr.  \and S. J. Gates, Jr. \and T. H{\" u}bsch \and R. Nath, \at Queensborough Community College, Bayside, New York\\ email: RDonley@qcc.cuny.edu\\ \\ S. J. Gates, Jr., e-mail: gatess@umd.edu\\ \\ T. H{\" u}bsch, e-mail: Thubsch@Howard.edu \\  \\ R. Nath, email:  rnath@york.cuny.edu}
%
%
\maketitle

\abstract*{We survey the combinatorics of the Adinkra, a graphical device for solving differential equations in supersymmetry.  These graphs represent an exceptional class of 1-factorizations with further augmentations.  As a new feature, we characterize Adinkras using Latin rectangles.}

\abstract{We survey the combinatorics of the Adinkra, a graphical device for solving differential equations in supersymmetry.  These graphs represent an exceptional class of 1-factorizations with further augmentations.  As a new feature, we characterize Adinkras using Latin rectangles.}


\section{Introduction}
\label{sec:1}

Twenty years ago, M. Faux and S. J. Gates, Jr. \cite{FGa} defined the Adinkra, a graphical device for enabling computations in supersymmetry.  In the Adinkra model, an equal number of bosons and fermions are interchanged under supercharges in many different ways, and the constraints on these ways lead to interesting extensions of concepts from matching theory. 

Though the authors of the inaugural presentation in the physics literature were unaware of it, a mathematical precursor and wellspring for this subject occurred in the 1960s. Mathematicians J.-M. L{\' e}vy-Leblond \cite{JLL} and N. Sen Gupta \cite{NSG} studied the representation theory in the limit of a specific Wigner-In{\" o}n{\" u} contraction of the spacetime Lorentz group. This limit was ultimately given the name of the ``Carroll Limit."  It has been noted by K. Koutrolikos and M. Najafizadeh \cite{Kou} that Adinkras emerge whenever the Carroll Limit is applied to any supersymmetrical theory defined in any Lorentz representation context.

In a recent comment, the physicist E. Witten implied that physics theories that are interesting tend to accrete very interesting mathematics about them. For adinkras, there are several existing examples of this happening.  However, the most startling occurrence of this phenomenon is the proof \cite{DI1}, \cite{DI2} that Adinkras are in fact examples of Grothendieck's ``dessin d'enfant" \cite{G} in the modern vernacular.  That is, Adinkras define combinatorial objects with the property of allowing the study of branched covers that relate two mathematical concepts:  Galois groups in the rational numbers, and Riemann surfaces. 

The goal of this work is to survey the combinatorial aspects of Adinkras.  Much of the material in this work is well-known from the literature; a novel feature is the representation and characterization of Adinkras using Latin rectangles for their adjacency lists.   The contents are arranged by section as follows:

\begin{enumerate}[noitemsep]
\item[2.]  1-factorizations and Regular Edge Colorings,
\item[3.]  The Quadrilateral Property,
\item[4.]  Adjacency Lists and Latin Rectangles,
\item[5.]  Adjacency Matrices and Semi-magic Squares,
\item[6.]  The Hypercube Graphs $Q_N$ and their Quotients,
\item[7.] Totally Odd Dashings,
\item[8.] Partial Orderings, and
\item[9.] Motivation from Supersymmetry
\end{enumerate}

Further mathematical treatments of the subject are found in \cite{DIL}, \cite{Ig1}, \cite{IKK}, \cite{ZhY}.  

\section{1-factorizations and Regular Edge Colorings}
\label{sec:2}

Let $G$ be a finite simple graph with vertex set $V$ and edge set $E$.   For language conventions from graph theory, walks can revisit vertices; paths cannot.   A circuit is a walk that begins and ends on the same vertex; a cycle is is a circuit that does not repeat vertices otherwise.

We begin with a definition whose conditions will be the subject of later sections.

\begin{definition}  An {\bf Adinkra} is a bipartite graph $G$ with
\begin{itemize}[nolistsep]
\item a regular edge coloring,
\item the quadrilateral property,
\item a totally odd dashing on the edges, and 
\item a compatible partial ordering.
\end{itemize}
\end{definition}

The first part of the definition, the regular edge coloring, expresses the property of a 1-factorization on $G$ using colored edges.

\begin{definition} A {\bf perfect matching} $M$ in $G$ is a matching that covers every vertex of $G$. That is, there is a subset $M$ of $E$ such that every vertex in $V$ is adjacent to exactly one edge in $M$.
\end{definition}

In the context of spanning subgraphs, a perfect matching is called a 1-{\bf factor.}
\vspace{5pt}

\begin{definition}  A 1-{\bf factorization} of $G$ is a partition of the edge set $E$ into disjoint 1-factors.  If we assign a color to each subset of the partition, we obtain a {\bf regular edge coloring} of $G$.  That is, each vertex meets exactly one edge of a given color.
\end{definition}

It will be useful to have a precise definition for an edge coloring of $G$.

\begin{definition} A {\bf edge coloring} $c$ on a graph $G$ with $N$ colors is a surjective function $c: E\to \{1, \dots, N\}$. We also use this function to order the colors.
\end{definition}

A graph with a perfect matching must have an even number of vertices, and a regular edge coloring with $N$ colors implies the graph is $N$-regular.  Examples of graphs with 1-factorizations include
\begin{itemize}[noitemsep]
\item cycles of even length,
\item complete graphs $K_{2n}$ with an even number of vertices,
\item complete bipartite graphs $K_{n, n}$ with parts of equal size, and
\item hypercube graphs $Q_n$ and their quotients.
\end{itemize}

For the second example, we place the vertices of  $K_{2n}$ in the shape of a regular $(2n-1)$-gon with a vertex at the center.  Each radial edge is assigned a distinct color, and the remaining edges with a given color are those perpendicular to the radial edge of that color. In Fig 1, we see the regular coloring on $K_4$, the 1-skeleton of the tetrahedron, with its associated perfect matchings.  

\setcounter{figure}{0}
\begin{figure}[h]
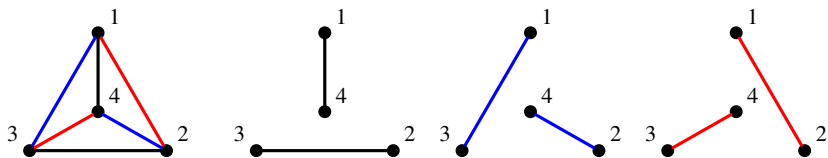

\TikZ{[every node/.style={inner sep=.5pt, outer sep=0, shape=circle, 
                        draw=black, label distance=1mm},
        very thick, scale=1.4]
       \path[use as bounding box](-.2,-.2)--(1.9,1.7);
       \node[fill=black, label={ 45:$1$}](1) at(.65,1.12){\phantom{.}}; 
       \node[fill=black, label={ 45:$2$}](2) at(1.3,0){\phantom{.}};
       \node[fill=black, label={ 135:$3$}](3) at(0,0){\phantom{.}};
       \node[fill=black, label={ 45:$4$}](4) at(.65,.37){\phantom{.}};
       \draw[black](2)--(3);
       \draw[black](1)--(4);
       \draw[red](1)--(2);
       \draw[red](3)--(4);
       \draw[blue](1)--(3);
       \draw[blue](2)--(4);      
       } 
\TikZ{[every node/.style={inner sep=.5pt, outer sep=0, shape=circle, 
                           draw=black, label distance=1mm},
        very thick, scale=1.4]
       \path[use as bounding box](-.2,-.2)--(1.7,1.7);
       \node[fill=black, label={ 45:$1$}](1) at(.65,1.12){\phantom{.}}; 
       \node[fill=black, label={ 45:$2$}](2) at(1.3,0){\phantom{.}};
       \node[fill=black, label={ 135:$3$}](3) at(0,0){\phantom{.}};
       \node[fill=black, label={ 45:$4$}](4) at(.65,.37){\phantom{.}};
       \draw[black](2)--(3);
       \draw[black](1)--(4);
       }    
 \TikZ{[every node/.style={inner sep=.5pt, outer sep=0, shape=circle, 
                           draw=black, label distance=1mm},
        very thick, scale=1.4]
       \path[use as bounding box](-.2,-.2)--(1.7,1.7);
       \node[fill=black, label={ 45:$1$}](1) at(.65,1.12){\phantom{.}}; 
       \node[fill=black, label={ 45:$2$}](2) at(1.3,0){\phantom{.}};
       \node[fill=black, label={ 135:$3$}](3) at(0,0){\phantom{.}};
       \node[fill=black, label={ 45:$4$}](4) at(.65,.37){\phantom{.}};
       \draw[blue](1)--(3);
       \draw[blue](2)--(4);
       }    
\TikZ{[every node/.style={inner sep=.5pt, outer sep=0, shape=circle, 
                           draw=black, label distance=1mm},
        very thick, scale=1.4]
       \path[use as bounding box](-.2,-.2)--(1.7, 1.7);
       \node[fill=black, label={ 45:$1$}](1) at(.65,1.12){\phantom{.}}; 
       \node[fill=black, label={ 45:$2$}](2) at(1.3,0){\phantom{.}};
       \node[fill=black, label={ 135:$3$}](3) at(0,0){\phantom{.}};
       \node[fill=black, label={ 45:$4$}](4) at(.65,.37){\phantom{.}};
       \draw[red](1)--(2);
       \draw[red](3)--(4);
       }      
\caption{A regular edge coloring for $K_4$ with its associated perfect matchings}
\end{figure}

For the third example, suppose $V$ partitions into disjoint subsets $V_1=\{v_i\}$ and $V_2=\{w_i\}$, both with $n$ elements.  With an ordering on the $n$ colors, we attach an edge of the first color from $v_1$ to $w_1$, and continue to attach edges from $v_1$ to $w_i$ clockwise in order of color.  For $v_2$, we attach an edge of the second color to $w_1$, continue clockwise as for $v_1$, and then repeat for all other $v_i$. In Fig 2, we cycle clockwise through black, blue, green, and red. 

\setcounter{figure}{1}
\begin{figure}[h]
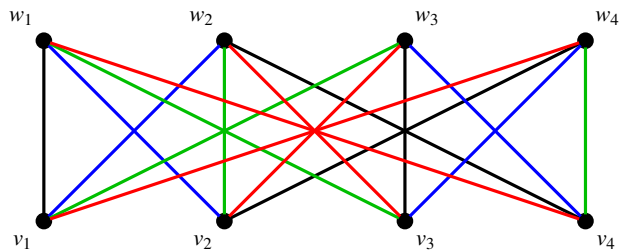

 \centering
 \TikZ{[every node/.style={inner sep=1pt, outer sep=0, shape=circle, 
                           very thick, draw=black, label distance=1mm}, very thick, scale=.8]
       \path[use as bounding box](-3.2,-.4)--(3.2,3.5);
       \node[fill=black, label={225:$v_1$}](1) at(-4.5,0){\phantom{.}};
       \node[fill=black, label={225:$v_2$}](2) at(-1.5,0){\phantom{.}};
       \node[fill=black, label={315:$v_3$}](3) at( 1.5,0){\phantom{.}};
       \node[fill=black, label={315:$v_4$}](4) at( 4.5,0){\phantom{.}};
       \node[fill=black, label={135:$w_1$}](5) at(-4.5,3){\phantom{.}};
       \node[fill=black, label={135:$w_2$}](6) at(-1.5,3){\phantom{.}};
       \node[fill=black, label={ 45:$w_3$}](7) at( 1.5,3){\phantom{.}};
       \node[fill=black, label={ 45:$w_4$}](8) at( 4.5,3){\phantom{.}};
       \draw[black](2)--(8);
       \draw[black](3)--(7);
       \draw[black](1)--(5);
       \draw[black](4)--(6);      
       \draw[blue](2)--(5);
       \draw[blue](3)--(8);
       \draw[blue](1)--(6);
       \draw[blue](4)--(7);
       \draw[green!75!black](2)--(6);
       \draw[green!75!black](3)--(5);
       \draw[green!75!black](1)--(7);
       \draw[green!75!black](4)--(8);
       \draw[red](2)--(7);
       \draw[red](3)--(6);
       \draw[red](1)--(8);
       \draw[red](4)--(5);
       }
 \caption{A regular edge coloring for $K_{4, 4}$}
\end{figure}

For the last item in the list, a partial classification of Adinkras includes a process for constructing regular edge colorings from quotients of hypercubes. We describe the quotienting process in Section 6.

In general, an elementary criterion for the existence of a 1-factorization follows from a corollary to Hall's marriage theorem \cite{Ha}.

\begin{corollary} If a graph $G$ is bipartite and $N$-regular, then $G$ admits a 1-factorization.
\end{corollary}

To implement the corollary, it will be helpful to recall

\begin{proposition} A graph $G$ is bipartite if and only if all cycles in $G$ have even length.
\end{proposition}

For the 1-skeletons of of the Platonic solids, the corollary only applies to the cube, and, in general, it applies to all hypercube graphs $Q_n.$  Without the bipartite condition, it is known that 1-factorizations always exist for $N$-regular graphs for sufficiently large $N$.  We note some general related conjectures.  See \cite{CH} and more recently \cite{CKL}. 

\begin{conjecture}{(The 1-factorization conjecture)}  Suppose $G$ is a $N$-regular graph with $2n$ vertices. For $n$ odd, $G$ has a 1-factorization if $N\ge n$. For even $n$, $G$ has a 1-factorization if $N \ge n-1.$
\end{conjecture}

Recall that a cycle in a graph $G$ is called {\bf Hamiltonian} if it meets every vertex.

\begin{conjecture}{(The Hamiltonian decomposition conjecture)} Suppose $N\ge \lfloor \frac{n}2\rfloor$. Then every $N$-regular graph $G$ on $n$ vertices has a decomposition into Hamiltonian cycles and at most one perfect matching.
\end{conjecture}

Again, such decompositions hold for complete graphs $K_{2n}$ (Hamiltonian decomposition) and $K_{2n+1}$ (Hamiltonian decomposition plus perfect matching) . See \cite{Als} for Walecki's constructions in these cases.

For the general theory of matchings, see \cite{LPl}.  For an overview of significant results and an extensive bibliography of graph factorizations, see \cite{Pl1} and \cite{Pl2}.

\section{The Quadrilateral Property}
\label{sec:3}

Suppose $G$ has a regular edge coloring $c$ with $N$ colors.   A key concept of supersymmetry is the exchange of bosons and fermions, and, when the physical model is a graph, this exchange is implemented by exchanging the labels on the vertices of the edges in a perfect matching.  With a  regular edge coloring, several options for such an exchange exist. 

\begin{definition}  Suppose $V=\{v_1, \dots, v_n\}.$ To each color $t$, we define a permutation $s_t$ in the symmetric group $S_n$ such that $s_t(i)=j$ if $v_iv_j$ is an edge with color $t$.
\end{definition}

Each $s_t$ is an involution and a derangement; that is, $s_t^2=e$, and $s_t$ has no fixed points.

\begin{definition} The {\bf exchange group} $Ex(G, c)$ (or $Ex(c)$) of a regular edge coloring $c$ is the subgroup of $S_n$ generated by the $s_t.$
\end{definition}

\begin{remark} Elements of $Ex(c)$ are typically not graph isomorphisms, but instead permute the vertex labels of $G$.  However, for the hypercube graphs $Q_N$ in Section 6, elements of $Ex(c)$ can be interpreted as graph isomorphisms.

More precisely, if we assign position numbers to the vertices, then each vertex labeling can be represented by a string of digits where the label for vertex $i$ is in the $i$-th position of the string. If $s_t(i)=j$ then $s_t$ replaces the label in position $j$ with the label in position $i$.  If we consider the string as a permutation in one-line notation, then permutations in $Ex(c)$ act on elements of the the symmetric group on the right with an inverse.   The right weak Bruhat order on the symmetric group uses the same action with transpositions.

 In Fig 1, suppose black, blue, and red are assigned 1, 2, and 3, respectively, and we label vertex positions 1, 2, 3, and 4 with $A$, $B$, $C$, and $D$, respectively. In cycle notation, $s_1=(14)(23),\ s_2=(13)(24),$ and $s_3=(12)(34)$. Then, for instance, 
$$(s_1s_2)\cdot ABCD =  s_1\cdot (s_2\cdot ABCD)=s_1\cdot CDAB=BADC=s_3\cdot ABCD.$$
\end{remark}

\begin{example} Let $G$ be a {\bf bicolor cycle}, that is, a cycle of length $n=2m$ with two colors. With a clockwise numbering such that $v_1v_2$ has color 1, $s_1=(12)(34)\dots (n-1\ n)$ and $s_2 = (23)\dots (1n).$ Then $s_1^2=s_2^2=e$ and $(s_1s_2)^m=e$.  Thus $Ex(c)=D_{2m},$ the dihedral group with $2m$ elements.
\end{example}

We recall

\begin{definition}  A Coxeter group $W(\{s_i\})$ is defined by
\begin{itemize}[noitemsep]
\item a generating set of reflections $\{s_i\}$,
\item a set of relations $(s_is_j)^{m_{ij}}=e$ where $m_{ii}=1$ and $m_{ij}\ge 2,$ possibly infinite, and
\item no other relations.
\end{itemize}
\end{definition}

Since the defining involutions for $Ex(c)$ satisfy the first two properties, $Ex(c)$ is a quotient of a corresponding Coxeter group. Since each $s_t$ is in $S_n$, we assume each $m_{ij}$ is finite.

We consider the two extremes for $m_{ij}$.  Consider the subgraph $G_{ij}$ consisting of the edges with colors $i$ and $j.$ Then $G_{ij}$ is an union of disjoint bicolor cycles of even length.  If these lengths are given by $\{2l_1, \dots, 2l_t\}$ then $m_{ij} = lcm(l_1, \dots, l_t),$ where $lcm$ denotes the least common multiple.

Suppose $G$ has an even number $n$ of vertices with $n=2m$.  If some $m_{ij}=m$, then the cycle associated to the colors $i$ and $j$ meets every vertex and is a {\bf Hamiltonian cycle}.  If every $m_{ij}=m$ when $i\ne j$ then every bicolor cycle is Hamiltonian, and the regular edge coloring is called a {\bf perfect 1-factorization}.   Fig 3 gives an example of a perfect 1-factorization; in Fig 2, the black and green edges fail to give a Hamiltonian cycle.
\vspace{5pt}

\setcounter{figure}2
\begin{figure}[h]
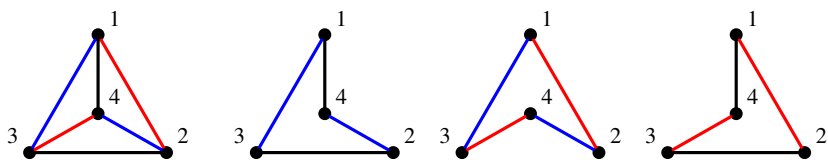

\TikZ{[every node/.style={inner sep=.5pt, outer sep=0, shape=circle, 
                        draw=black, label distance=1mm},
        very thick, scale=1.4]
       \path[use as bounding box](-.2,-.2)--(1.9,1.7);
       \node[fill=black, label={ 45:$1$}](1) at(.65,1.12){\phantom{.}}; 
       \node[fill=black, label={ 45:$2$}](2) at(1.3,0){\phantom{.}};
       \node[fill=black, label={ 135:$3$}](3) at(0,0){\phantom{.}};
       \node[fill=black, label={ 45:$4$}](4) at(.65,.37){\phantom{.}};
       \draw[black](2)--(3);
       \draw[black](1)--(4);
       \draw[red](1)--(2);
       \draw[red](3)--(4);
       \draw[blue](1)--(3);
       \draw[blue](2)--(4);      
       } 
\TikZ{[every node/.style={inner sep=.5pt, outer sep=0, shape=circle, 
                           draw=black, label distance=1mm},
        very thick, scale=1.4]
       \path[use as bounding box](-.2,-.2)--(1.7,1.7);
       \node[fill=black, label={ 45:$1$}](1) at(.65,1.12){\phantom{.}}; 
       \node[fill=black, label={ 45:$2$}](2) at(1.3,0){\phantom{.}};
       \node[fill=black, label={ 135:$3$}](3) at(0,0){\phantom{.}};
       \node[fill=black, label={ 45:$4$}](4) at(.65,.37){\phantom{.}};
       \draw[black](2)--(3);
       \draw[black](1)--(4);
       \draw[blue](1)--(3);
       \draw[blue](2)--(4);
       }     
 \TikZ{[every node/.style={inner sep=.5pt, outer sep=0, shape=circle, 
                           draw=black, label distance=1mm},
        very thick, scale=1.4]
       \path[use as bounding box](-.2,-.2)--(1.7,1.7);
       \node[fill=black, label={ 45:$1$}](1) at(.65,1.12){\phantom{.}}; 
       \node[fill=black, label={ 45:$2$}](2) at(1.3,0){\phantom{.}};
       \node[fill=black, label={ 135:$3$}](3) at(0,0){\phantom{.}};
       \node[fill=black, label={ 45:$4$}](4) at(.65,.37){\phantom{.}};
       \draw[blue](1)--(3);
       \draw[blue](2)--(4);
       \draw[red](1)--(2);
       \draw[red](3)--(4);
       }
\TikZ{[every node/.style={inner sep=.5pt, outer sep=0, shape=circle, 
                           draw=black, label distance=1mm},
        very thick, scale=1.4]
       \path[use as bounding box](-.2,-.2)--(1.7, 1.7);
       \node[fill=black, label={ 45:$1$}](1) at(.65,1.12){\phantom{.}}; 
       \node[fill=black, label={ 45:$2$}](2) at(1.3,0){\phantom{.}};
       \node[fill=black, label={ 135:$3$}](3) at(0,0){\phantom{.}};
       \node[fill=black, label={ 45:$4$}](4) at(.65,.37){\phantom{.}};
       \draw[red](1)--(2);
       \draw[red](3)--(4);
       \draw[black](2)--(3);
       \draw[black](1)--(4);
       }      
\caption{The quadrilateral property for $K_4$}
\end{figure}

\begin{conjecture}{(Perfect 1-factorization Conjecture \cite{Ko})}  Every complete graph $K_{2n}$ with an even number of vertices admits a perfect 1-factorization.
\end{conjecture}

In general, the conjecture is true for $n=p$ and $2n=p+1$ when $p$ is an odd prime.  Otherwise, the smallest open case is $K_{64}$. 

At the other extreme, suppose all $m_{ij}=2$ when $i\ne j.$ 

\begin{definition} A regular edge coloring $c$  on $G$ satisfies the {\bf quadrilateral property} if, for every $i$ and $j$, the subgraph $G_{ij}$ with edges for colors $i$ and $j$  is a disjoint union of 4-cycles.
\end{definition}

\begin{proposition} (\cite{DFG}) Let $c$ be a regular edge coloring on $G$. The following statements are equivalent:
\begin{enumerate}[noitemsep]
\item each $m_{ij}=2$ when $i\ne j,$
\item $Ex(c)$ is abelian and isomorphic to a product of $\mathbb{Z}/2$ groups, and 
\item  $c$ satisfies the quadrilateral property. 
\end{enumerate}
\end{proposition}

\begin{proof}  For the equivalence of the first two statements, if $m_{ij}=2,$ then $(s_is_j)^2=e$. Since each $s_i=s_i^{-1},$ $s_is_j=s_js_i.$ The converse follows similarly.

For the equivalence of the first and third statements, if each $m_{ij}=2$ when $i\ne j,$ then all cycle lengths in $G_{ij}$ have length 4.  Again, the converse holds immediately.  $\square$
\end{proof}

\begin{example}  Consider the regular edge coloring of $K_4$ in Fig 1.  Then $Ex(C)\cong \mathbb{Z}/2\times \mathbb{Z}/2.$  Each $G_{ij}$ consists of a single 4-cycle, as seen in Fig 3.  Note that, while $Ex(C)$ has three generators, $s_1s_2s_3=e.$
\end{example}

\begin{example}  Consider the regular edge coloring of the cubical graph $Q_3$ in Fig 4.  Suppose blue, green, and red are assigned the numbers 1, 2, and 3, respectively.  Then the bicolor cycles consist of two Hamiltonian 8-cycles and a pair of 4-cycles. Thus $Ex(c)$ is non-abelian, and $s_1$ and $s_2$ commute. By observation, $s_1s_3 = s_3s_2.$  The subgroup $H$ generated by $s_1$ and $s_2$ is a normal subgroup, and $Ex(c)$ is a semi-direct product of $H$ by the $\mathbb{Z}/2$-subgroup containing $s_3.$ Since $Ex(C)$ has 8 elements and more than one element of order 2, $Ex(c)\cong D_8,$ the dihedral group with 8 elements.
\end{example}

\setcounter{figure}3
\begin{figure}[h]
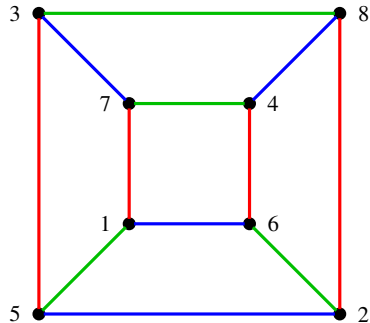

 \centering
 \TikZ{
       [every node/.style={inner sep=.5pt, outer sep=0, shape=circle, 
                           very thick, draw=black, label distance=1mm},
        very thick, yscale=1]
       \path[use as bounding box](-.2, -.2)--(4.5,4.5);
       \node[fill=black, label={180:1}](1) at( 1.2, 1.2){\phantom{.}};
       \node[fill=black, label={180:7}](7) at( 1.2, 2.8){\phantom{.}};       
       \node[fill=black, label={  0:6}](6) at( 2.8, 1.2){\phantom{.}};
       \node[fill=black, label={  0:4}](4) at( 2.8, 2.8){\phantom{.}};       
       \node[fill=black, label={  0:2}](2) at( 4, 0){\phantom{.}};
       \node[fill=black, label={  180:3}](3) at(0, 4){\phantom{.}};
       \node[fill=black, label={  180:5}](5) at( 0, 0){\phantom{.}};
       \node[fill=black, label={  0:8}](8) at( 4, 4){\phantom{.}};
       \draw[red](1)--(7);
       \draw[red](3)--(5);
       \draw[red](4)--(6);
       \draw[red](2)--(8);
       \draw[blue](1)--(6);
       \draw[green!75!black](7)--(4);
       \draw[blue](5)--(2);
       \draw[green!75!black](3)--(8);
       \draw[green!75!black](1)--(5);
       \draw[blue](7)--(3);
       \draw[green!75!black](6)--(2);
       \draw[blue](4)--(8);
       }
 \caption{A regular edge coloring for $Q_3$}
\end{figure}

\section{Adjacency Lists and Latin Rectangles}
\label{sec:4}

Adjacency matrices play a role for various methods on Adinkras.  If we consider the adjacency list for a graph instead, edge color regularity results in Latin rectangles. See \cite{DK} and \cite{Sto} for background and basic properties of Latin rectangles. For 1-factorizations on complete bipartite graphs using Latin squares, see \cite{BMW} and \cite{IW}.

\begin{definition}  Let $S$ be a set of $n$ distinct symbols.  A {\bf Latin rectangle} with $m$ rows is an $m\times n$ rectangular array with entries in $S$ such that each element of $S$ occurs exactly once in each row and at most once in each column.
\end{definition}

Recall that the adjacency list of a graph $G$ is a table with $n$ columns such that each column corresponds to a vertex of $G$. The entries in the column for vertex $v$ are those vertices joined to $v$ by an edge, in no particular order.  The graph can be reconstructed from the adjacency list.
\vspace{5pt}

\begin{proposition} (Adjacency List Property)  Let $S$ be a set of $n$ distinct symbols. Suppose we have a table $T$ with $n$ columns, one for each element of $S$, and each symbol occurs at most once in each column.   Then $T$ is an adjacency list for a graph $G$ with $n$ vertices if and only if, whenever $v$ is in the column for $w$, $w$ is in the column for $v.$
\end{proposition}

If $G$ is $N$-regular, then the adjacency list is a $(N+1)\times n$ rectangle, including the column labels.  Suppose  we have a regular edge coloring $c$ on $G.$ If we assign the $i$-th color to the $i$-th row below the column labels, then $w$ is in row $i+1$ and column $j$ if and only if $w$ and $v_j$ are joined by an edge of color $i.$  Furthermore, if the first line represents $v_1,\dots, v_n$ in order, then the row for color $i$ represents the permutation $s_i$ using  one-line notation. 

\begin{proposition} Let $G$ be a graph with $n$ vertices, and suppose $c$ is a regular edge coloring on $G$ with $N$ colors.  Then the adjacency list $L(c)$ is a $(N+1) \times n$ Latin rectangle. 
\end{proposition}

Given a Latin rectangle with first line as column labels, we can assign a graph and regular edge coloring when the Adjacency List Property holds.  For the examples of $K_4$ and $Q_3$ above, $L(c)$ is given in Fig 5.

\begin{figure}
\centering
{\normalsize
$\begin{array}{ccccc}
V\ \ & \fbox{1} & \ 2\  & \fbox{3} & \ 4\ \\ \hline
Black  & 4 & 3 & 2 & 1\\
Blue & \fbox{3} & 4 & \fbox{1} & 2 \\
Red & 2 & 1 & 4  & 3 \\
\end{array}\qquad
\begin{array}{ccccc|cccc}
V\ \ & \ 1\  &\  2\  &\  3\  &\  \ 4 \ & \ 5\  &\  6\  &\  7\  &\  8\ \\ \hline
Blue & \fbox{6} & \fbox{5} & 7 & 8 & \fbox{2} & \fbox{1} & 3 & 4\\
Green & \fbox{5} & \fbox{6} & 8  & 7 & \fbox{1} & \fbox{2} & 4 & 3 \\
Red  & 7 & 8 & 5 & 6 & 3 & 4 & 1 & 2 \\
\end{array}$}
\caption{Adjacency lists $L(c)$ for $K_4$ and $Q_3$}
\end{figure}

Properties of the graph or the coloring can be read from $L(c)$:
\begin{enumerate}[noitemsep]
\item{\bf Adjacency:} If $v_i$ occurs in the column for $v_j$, then the corresponding entries in these columns form the corners of a rectangle with the column labels,
\item {\bf Connectedness}:   The columns of $L(c)$ cannot be partitioned into two adjacency lists, 
\item {\bf Bipartite}:  There is a bipartition of columns $C_1\cup C_2$ such that only the column labels for $C_2$ occur as entries in the columns of $C_1$, and vice versa, and
\item {\bf Quadrilateral Property}:  If we restrict $L(c)$ to two rows below the first line, entries in a given column will occur inverted in another column, forming a rectangle.  A second rectangle arises with these entries and column labels reversed.
\end{enumerate}

To see the fourth property, let $v_i$ and $v_j$ be the entries in rows $i$ and $j$ and in the column for $v_l$.  These three vertices form a bicolor 4-cycle with a fourth vertex $v_t$. Since the edge colors alternate in the 4-cycle, $v_j$ appears in row $i$ in the column for $v_t,$ and similarly for $v_i.$ By interchanging $v_i$ and $v_j$ with $v_l$ and $v_t$, we obtain the second rectangle. 

The Latin rectangle for $K_4$ gives an example of the adjacency property, exhibited by the boxed entries.   The rectangle  $L(Q_3, c)$ satisfies the bipartite property with $C_1=\{1, 2,3 ,4\}$ and $C_2=\{5, 6, 7, 8\}.$ Only the blue and green edges in this example satisfy the quadrilateral property;  the boxed entries form a pair of rectangles associated to a bicolor 4-cycle.

\section{Adjacency Matrices and Semi-magic Squares}
\label{sec:5}

An equivalent alternative to Latin rectangles arises if we instead consider semi-magic squares by realizing permutations through permutation matrices.  In this manner, the triple \{one-line notation, Latin rectangle, adjacency list\} corresponds to the triple \{permutation matrix, semi-magic square, adjacency matrix\}.

We first consider the adjacency matrix of a perfect matching.  Recall that the adjacency matrix of a simple graph $G$ on $n$ vertices is a square matrix of size $n$ whose nonzero entries $a_{ij}$ equal 1 if and only if $v_i$ and $v_j$ are connected by and edge.  Adjacency matrices of simple graphs are always symmetric with zeros on the main diagonal.
\vspace{5pt}

\begin{proposition}  Let $A_M$ be the adjacency matrix for a perfect matching $M$.  Then $A_M=P_\sigma,$ the permutation matrix associated to some permutation $\sigma$, where $\sigma$ is both an involution and a derangement. That is, $\sigma^2=e$ and $\sigma$ has no fixed points.
\end{proposition}

\begin{proof}  The adjacency matrix $A_M$ of a perfect matching has exactly one 1 in each row and column, has zeros on the main diagonal, and is symmetric.  Thus $A_M$ corresponds to some permutation matrix $P_\sigma$.  Since the inverse of a permutation matrix is its transpose, the symmetric property of $M$ implies $\sigma=\sigma^{-1}$ and $\sigma^2=e$. Since $P_\sigma$ has no 1s on the diagonal, $\sigma$ has no fixed points. $\square$
\end{proof}

\begin{definition} Suppose $G$ is a graph on $n$ vertices and $c$ is a regular edge coloring with $N$ colors.  The adjacency matrix $A_c$ of a regular edge coloring $c$ is the sum of the permutation matrices for each perfect matching, weighted by the color's number. 
\end{definition}

\begin{example}  Again we consider the adjacency lists for $K_4$ and $Q_3$ in Fig 5. Respectively, the adjacency matrices $A_c$ equal
$${\normalsize
\left[ \begin{array}{cccc}
\ 0\ & \ 3\ &\ 2 \ &\  1\ \\
\ 3\  &\  0\ &\ 1 \ & \ 2\ \\
\ 2\ & \ 1\ &\ 0 \ & \ 3 \ \\
\ 1\ & \ 2\ & \ 3 \ & \ 0 \
\end{array}\right],
\qquad
\left[\begin{array}{cccccccc}
 &  &  &  & \ 2\  & \ 1\  & \ 3\  & \ 0\ \\
  &  &\  0\  &  & \ 1\  & \ 2\  & \ 0\  &\  3\ \\
   &  &  &  & \ 3\  & \ 0\ &\  1 \ &\  2\ \\
    &  &  &  &\  0\  & \ 3 \ & \ 2\  &\  1\ \\
 \ 2\ & \ 1\  &\  3 \ & 0 & & & & \\
\ 1\ & \ 2\  & \ 0 \ &\  3\  & &\  0\  & & \\
 \ 3\ & \ 0\  & \ 1 \ &\  2\  & & & & \\
 \ 0\ & \ 3\  & \ 2 \ & \  1\  & & & & \\    
\end{array}\right].}
$$
\end{example}

Now $A_c$ is symmetric, and we obtain the adjacency matrix of $G$ by setting each nonzero entry in $A_c$ equal to 1.  In addition, each nonzero entry occurs exactly once in each row and column.  The number of nonzero entries in a row or column is the common degree $N$.

\begin{definition} Let $M$ be a square matrix of size $n$ with non-negative integer entries. We say that $M$ is a {\bf semi-magic square} with line sum $L$ if the sum along any row or column is equal to $L$.
\end{definition}

The classical definition of a magic square requires single entries of 1 through $n^2$ and that the common sum also holds along the main diagonals. Every sum of $m$ permutation matrices is a semi-magic square with line sum $m$, and, as a corollary to the Birkhoff-von Neumann theorem, every semi-magic square is of this form.

Since the rows and columns of $A_c$ contain the same values, we have

\begin{proposition} Let $G$ be a graph with $n$ vertices, and suppose $c$ is a regular edge coloring on $G$ with $N$ colors.  Then the adjacency matrix $A_c$ is a symmetric semi-magic square of size $n$. 
\end{proposition}

For a symmetric semi-magic square to correspond to a regular edge coloring on a graph, it must decompose by entry values into a sum of permutation matrices that are both symmetric and without zeros on the main diagonal.  We leave it to the reader to determine the heuristics for interpreting graph or coloring properties directly from $A_c$.  

Evidently, a graph $G$ with regular edge coloring $c$ can be reconstructed from its adjacency list $L(c)$ or adjacency matrix $A_c.$

\section{The Hypercube Graphs $Q_N$ and their Quotients}
\label{sec:6}

In this section, we define the hypercube graphs $Q_N$, describe a uniform coloring method for such graphs, describe the quotient of such a graph by a linear binary block code, and add a fourth equivalent condition to Proposition 2.  From this fourth condition, a partial classification of connected Adinkras is given in terms of doubly even codes,  perhaps one of the most striking results in the Adinkra literature.  We omit most proofs on quotients, referring the reader to \cite{DFG}.

\begin{definition} For the {\bf hypercube graph} $Q_N$, the vertex set $V(Q_N)=(\mathbb{Z}/2)^N.$ That is, a vertex of $Q_N$ is an $N$-tuple of binary numbers. Edges of $Q_N$ are defined between vertices that differ in exactly one entry; that is, edges are placed between vertices at Hamming distance 1.
\end{definition}

Evidently, $Q_N$ is connected.  By elementary counting, $Q_N$ has $2^N$ vertices and $N2^{N-1}$ edges. With our definition, $Q_N$ is bipartite and $N$-regular, so regular edge colorings on $Q_N$ exist. It will be convenient to denote $N$-tuples as bitstrings. For instance, we write $(1, 0, 1)$ as 101.

It is typical to realize $Q_N$ as a subset of $\mathbb{R}^N$, but the group structure on the vertices is central to our constructions. We define the standard basis vector $e_i$ in $V(Q_N)$ as the bitstring of all zeros except for a 1 in position $i$.  Since vertices for a given edge differ in exactly one position, this edge has vertices $v$ and $v+e_i$ for some unique $i$.  

\begin{definition} We define the {\bf parallel edge coloring} $c_N$ on $Q_N$ as follows:  we color an  edge $vw$ with color $i$ if and only if $w=v+e_i$.  
\end{definition}

\setcounter{figure}5
\begin{figure}[h]
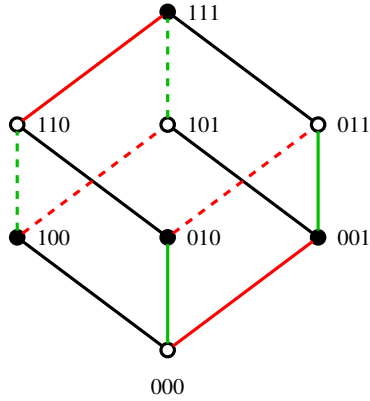

 \centering
 \TikZ{
       [every node/.style={inner sep=1pt, outer sep=0, shape=circle, 
                           very thick, draw=black, label distance=1mm},
        very thick, yscale=.75]
       \path[use as bounding box](-2.2,-.5)--(2.2,6.2);
       \node[fill=white, label={270:000}](1) at(0,0){\phantom{.}};
       \node[fill=white, label={  0:110}](5) at(-2,4){\phantom{.}};
       \node[fill=white, label={  0:101}](6) at( 0,4){\phantom{.}};
       \node[fill=white, label={  0:011}](7) at( 2,4){\phantom{.}};
       \node[fill=black, label={  0:100}](2) at(-2,2){\phantom{.}};
       \node[fill=black, label={  0:010}](3) at( 0,2){\phantom{.}};
       \node[fill=black, label={  0:001}](4) at( 2,2){\phantom{.}};
       \node[fill=black, label={  0:111}](8) at( 0,6){\phantom{.}};
       \draw[black](1)--(2);
       \draw[black](3)--(5);
       \draw[black](4)--(6);
       \draw[black](7)--(8);
       \draw[red](1)--(4);
       \draw[red, dashed](2)--(6);
       \draw[red, dashed](3)--(7);
       \draw[red](5)--(8);
       \draw[green!75!black](1)--(3);
       \draw[green!75!black, dashed](2)--(5);
       \draw[green!75!black](4)--(7);
       \draw[green!75!black, dashed](6)--(8);
       }
 \caption{A parallel edge coloring on $Q_3$ with dashings}
\end{figure}

When $N=2$ or 3, the definition evidently partitions $E$ into subsets of parallel edges of the same color; see Fig 6 for the parallel coloring on the cube. Furthermore, $c_N$ is a regular edge coloring, and each permutation $s_i$ simply adds $e_i$ to elements of $V(Q_N)$. Thus the permutations $s_i$ commute, and Proposition 2 holds.   

To see the quadrilateral property directly, we restrict to the subgraph $G_{ij}$ with edges of colors $i$ and $j$. The bicolor cycle in $G_{ij}$ containing vertex $v$ has vertex set of the form $v, v+e_i, v+e_j, v+e_i+e_j$ and is a 4-cycle with colors $i$ and $j$.  
\vspace{5pt}

\begin{remark}  In representation theoretic terms, $Q_N$ is a weight diagram for a root system of type $B_N$.  In this case, the Lie algebra $\mathfrak{so}(2N+1)$ acts on $S=\mathbb{C}^{2^N}$ by the spin representation. With the standard choices for roots $\Delta=\{e_i-e_j\ | \ i\ne j\},$ positive roots $\Delta^+=\{e_i-e_j, e_N\ | \ i<j\},$ and simple roots $\Pi=\{ \alpha_i=e_i-e_{i+1}, \alpha_N=e_N\},$  the weight set for $S$ consists of all weights $\Gamma=\{\frac12 (\pm e_1 \pm \dots \pm e_N)\}$ with highest weight $\frac12(e_1+\dots e_N)$.   These weights correspond to the vertices of $Q_N.$ In this context, this use of $e_i$ is not the same as for bitstrings above. 

In the Weyl group $W(\Delta)$ of signed permutation matrices, the root reflection $s_{e_i}$ corresponding to the short root $e_i$ acts by sign changes on the $i$-th coordinate of a weight.  Under the correspondence between these weights and binary $N$-tuples, this action is exactly the permutation $s_i.$    A weight $\gamma$ forms an edge of color $i$ with the weight $s_{e_i}\gamma$.  The reflection for the long root $e_i-e_j$ acts by interchanging signs in positions $i$ and $j$, with a corresponding effect on $V(Q_N).$ Finally, since $\Gamma$ is an orbit under $W(\Delta)$, $S$ is a minuscule representation of $\mathfrak{so}(2N+1).$
\end{remark}

We summarize the construction for edge colorings on quotients of cubes from \cite{DFG}. First, we review the basic terminology associated to binary block codes. We recall that $V(Q_N)$ is both a group under addition and a vector space over the field $\mathbb{Z}/2$. Denote the identity element by ${\bf 0}$.

\begin{definition} The weight $wt(x)$ of the bitstring $x$ is the number of entries equal to 1.  A bitstring $x$ is called {\bf even} if $wt(x)$ is even.
\end{definition}

\begin{definition} A {\bf binary block code} $C$ of length $N$ is a subset of $V(Q_N)$, and its elements are called {\bf codewords}.  If a code $C$ is also a subgroup of $V(Q_N),$ then $C$ is called a {\bf linear binary block code}, or {\bf linear code} for short.
\end{definition}

\begin{definition}. Let $C$ be a linear code in $V(Q_N).$ Then we define the quotient graph $Q_C$ as follows:
\begin{itemize}[noitemsep]
\item the vertex set  $V(Q_C)=V(Q_N)/C$; that is, vertices for $Q_C$ are cosets ${\bar v}=v+C$ with $v$ a vertex of $Q_N,$
\item an edge joins ${\bar v}$ and ${\bar w}$ if and only if  $v+w+e_i$ is in $C$ for some $i$, and,
\item in this case, the edge ${\bar v}{\bar w}$ has color $i$.
\end{itemize}
The resulting graph is a connected simple graph if $C$ has no codewords of weight 1 or 2.  With this condition, the induced coloring is also a regular edge coloring with the same colors.
\end{definition}

These definitions are well-defined; that is, they do not depend on the choices for $v$ and $w$ in the cosets.  If there is a loop with color $i$ at some vertex ${\bar v}$ in $Q_C,$ then $e_i=v+v+e_i$ is in $U.$  If there exist two edges between ${\bar v}$ and ${\bar w}$ with colors $i$ and $j$, $i\ne j$, then $v+w+e_i$ and $v+w+e_j$ are both in $C$, and $e_i+e_j$ is in $C$. Thus $Q_C$ has no loops or multiple edges with these conditions.  To see the regularity condition, for all $i$, ${\bar v}(\overline{v+e_i})$ is either a loop or a proper edge in $Q_C$ with color $i$.

\begin{example} Suppose $N\ge 3,$ and that $C$ has only two elements of length $N$, the identity ${\bf 0}$ and the codeword with all entries equal to 1. Then $Q_C=Q_{N}/C=F_N$, the folded cube graph.  In this formulation, $F_N$ is the cubic analogue of projective space, and an alternative construction for $F_N$ is to instead add to the diagonals to $Q_{N-1}$ with a new color. In particular, $F_3$ is the 1-skeleton of tetrahedron, and $F_4$ is the complete bipartite graph $K_{4, 4}.$ Of course, $F_N$ has $2^{N-1}$ vertices, $N$ colors, and $N2^{N-2}$ edges. Thus no $F_N$ is equivalent to a hypercube graph $Q_N.$
\end{example}

In general, if $\dim(C)=k$ and $C$ has no codewords of length 1 or 2 then $Q_C$ has $2^{N-k}$ vertices and $N2^{N-k-1}$ edges. 

\begin{definition} A linear code $C$ is called {\bf even} if every codeword in $C$ is even.
\end{definition}

If the linear code $C$ is given by a basis $B=\{u_1, \dots, u_k\},$ then it is sufficient that each $u_i$ be even for $C$ to be even.

\begin{proposition} (\cite{DFG}) Suppose $C$ has no codewords of weight 1 or 2. Then $Q_C$ is bipartite if and only if $C$ is an even linear code.
\end{proposition}

\begin{example} $F_N$ is bipartite if and only if $N$ is even.  We see this immediately in the alternative descriptions for $F_3$ and $F_4$.
\end{example}

\begin{proposition}(\cite{DFG}) $Q_C$ satisfies the quadrilateral property if and only if $C$ is a linear code with no codewords of length 1 or 2.
\end{proposition}

We now add another equivalent condition to Proposition 2.

\begin{proposition}(\cite{DFG}). Suppose $G$ is a connected simple graph with $n$ vertices, and $c$ is a regular edge coloring of $G$ with $N$ colors. If $c$ satisfies the quadrilateral property, then $G$  is graph isomorphic to a quotient $Q_C$ of the hypercube graph $Q_N$ by a linear code $C$ with no codewords of weight 1 or 2.  The regular edge coloring $c$ is induced from the parallel coloring $c_N$ on $Q_N.$ 
\end{proposition}

The key step of the proof, which we call the Walk Reduction Lemma below, uses the quadrilateral property to lift walks from $G$ to $Q_N$ consistently and guarantees a surjective map from $Q_N$.  More precisely, there is a direct correspondence between operations on walks and calculations in $Ex(c)$. Walks can be considered as sequences of colors from a base vertex, and in turn these sequences correspond to products of involutions in $Ex(c).$ 
\vspace{5pt}

\begin{lemma}(Walk Reduction Lemma, \cite{DFG}) Suppose $c$ satisfies the quadrilateral property, and $W$ is a walk in $G$ from $v$ to $w$.  Then there exists a walk $W'$ in $G$ from $v$ to $w$ such that each edge color occurs at most once.  Each ordering of these colors gives a distinct walk from $v$ to $w$.
\end{lemma}

\begin{proof} Since $c$ is regular, any walk from $v$ is described by a word in the colors, read from left to right. By the quadrilateral property, we can switch any adjacent colors in a word to obtain a new walk from $v$ to $w$.  On the other hand, adjacent pairs of the same color can be removed. With these two operations, we can reduce the walk to one where each color occurs at most once, and any ordering of the remaining colors can be used.  $\square$
\end{proof}

To find the code $C$, we fix a base vertex $v_0$ in $G$ and record the bitstrings associated with each cycle at $v_0$ in $G$ to find a basis. That is, we record the sequence of colors in the cycle, order, and cancel multiple colors in pairs. The code does not depend on the choice of $v_0.$ In $L(c),$ we generate cycles by choosing an entry $v_1$ in the column for $v_0,$ record the row color for $v_1$,  choose and entry in the column for $v_1$, and repeat until we return to column $v_0$. A decision tree can be used to enumerate all such cycles.

Consider $L(K_4)$ in Fig 5.  Suppose $v_0=1,$ and consider the cycle $1\to 2 \to 3 \to 1.$
The corresponding sequence of colors is Red, Black, Blue, which gives the codeword 111.  On the other hand, the bicolor 4-cycle $1\to 2 \to 3 \to 4 \to 1$ has color sequence Red, Black, Red, Black and codeword 000.  Of course, $C=\{000, 111\}$ since $K_4=F_3.$ 

\begin{remark} The equivalence in the proposition induces a multiplication on $V(G)$ from $V(Q_C)$.  This multiplication depends on $v_0$, and a similar type of multiplication with base point holds on elliptic curves. In fact, elliptic curves play an important role in the geometrization of Adinkras to Riemann surfaces.
\end{remark}

We conclude this section with a definition and examples.

\begin{definition} A {\bf pre-Adinkra} is a bipartite graph $G$ with regular edge coloring $c$ that satisfies the quadrilateral property.  In the graph, the bipartition is denoted by a vertex coloring with {\bf fermions} as closed circles and {\bf bosons} as open circles.
\end{definition}

Of course, a connected pre-Adinkra is isomorphic to $Q_C$ for some even code $C$. 

\begin{remark}  For notions of isomorphism, see to \cite{DFG}, especially the connection between graphs with colorings and linear codes.
\end{remark}

\begin{example}  We have the following isomorphism classes for connected pre-Adinkras.
\begin{itemize}[noitemsep]
\item $N=2$:  the 4-cycle with two colors,
\item $N=3$:  the cube $Q_3$ with the parallel edge coloring,
\item $N=4$:  the hypercube $Q_4$ with parallel edge coloring, or the folded cube $F_4$ with the coloring induced from $Q_4$.  The latter case is just the complete bipartite graph $K_{4, 4}$, and
\item $N=5$:  the hypercube $Q_5$, and the quotient $Q_C$ with $C=\{00000, 11110\}.$ The latter case is a prism over $K_{4, 4}$. See \cite{Ig1} for the definition of a prism for a regular edge coloring.
\end{itemize}
\end{example}

\section{Totally Odd Dashings}
\label{sec:7}

We consider the next condition in the definition of an Adinkra, the totally odd dashing.  This property considers $G$ as a signed graph. That is, to each edge, we assign a value of $1$ or $-1$.  To reduce notational clutter  from our diagrams, we indicate edges with sign $-1$ by using a dashed edge.

\begin{definition} Suppose the graph $G$ with regular edge coloring $c$ satisfies the quadrilateral property.  A {\bf totally odd dashing} on $G$ is an assignment of dashed edges to $G$ such that the number of dashed edges in each bicolor cycle is 1 or 3.
\end{definition}

\begin{example}  To put a totally odd dashing on the edges of $Q_N,$ we augment the group structure on $V(Q_N)$ with signs.  Instead of using bitstrings, we consider this group $Q'$ as generated by the standard basis vectors $e_i$ subject to the relations $$e_ie_i=1, \qquad e_ie_j=-e_je_i\ \ (i \ne j).$$ Then, for each $I=\{i_1, \dots, i_m\}\subseteq \{1, \dots, N\},$ every element in $Q'$ can be represented uniquely as $\pm g_I$, where
$$g_I= e_{i_1}\dots e_{i_m},\quad  0\le i_1 < \dots < i_m\le N.$$

To obtain a totally odd dashing, we consider products by each $e_i$.   This product is consistent with the parallel coloring; if $g_I, g_{I'}$ are in $V(Q_N)$ then $e_i g_I=\pm g_{I'}$ if and only $g_I$ and $g_{I'}$ differ by the factor $e_i$.  Additionally, we use a dashed edge if this sign is negative.

For example, consider the 4-cycle with vertices $1, e_i, e_j$ and $e_ie_j$ with $i<j$. Then only the edge from $e_i$ to $e_ie_j$ is dashed.   In general, if we ignore signs, a 4-cycle for colors $i$ and $j$ has vertices $g$, $e_ig$ $e_jg$ and $e_ie_jg$, and we can assume $i<j$ and $g$ has neither factor.  If $x$ is the number of indices for $g$ less than $i$ and $y$ is the number of indices for $g$ between $i$ and $j$, then the signs on the edges of the 4-cycle are $(-1)^x,$ $(-1)^{x+y},$ $(-1)^x$ and $(-1)^{x+y+1}$. Since the product is $-1$, an odd number of edges are dashed in the 4-cycle.   See Fig 6 for a dashing of this type.
\end{example}

In the adjacency list $L(c),$ we place negative signs on entries representing dashed edges.  Of course, if the entry for $v$ in the column for $w$ has a negative sign, so does the entry for $w$ in the column for $v$. The totally odd dashing is characterized by having 1 or 3 negative signs in the rectangle of boxes corresponding to a bicolor 4-cycle.  See Fig 7 for an example with its adjacency list $L(c)$. The boxed entries correspond to the red and black 4-cycle with vertices 2, 5, 6, and 8.

\setcounter{figure}6
\begin{figure}[h]
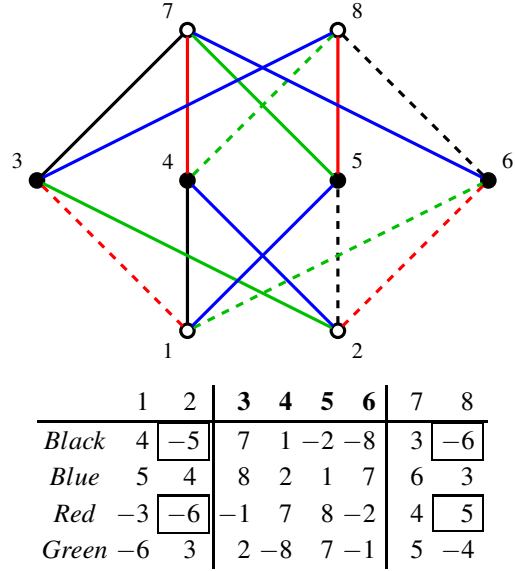

 \centering
 \TikZ{[every node/.style={inner sep=1pt, outer sep=0, shape=circle, 
                           very thick, draw=black, label distance=1mm}, very thick]
       \path[use as bounding box](-3.2,-.7)--(3.2,4.2);
       \node[fill=white, label={225:1}](1) at(-1,0){\phantom{.}};
       \node[fill=white, label={315:2}](2) at( 1,0){\phantom{.}};
       \node[fill=white, label={135:7}](7) at(-1,4){\phantom{.}};
       \node[fill=white, label={ 45:8}](8) at( 1,4){\phantom{.}};
       \node[fill=black, label={135:3}](3) at(-3,2){\phantom{.}};
       \node[fill=black, label={135:4}](4) at(-1,2){\phantom{.}};
       \node[fill=black, label={ 45:5}](5) at( 1,2){\phantom{.}};
       \node[fill=black, label={ 45:6}](6) at( 3,2){\phantom{.}};
       \draw[black](1)--(4);
       \draw[black, dashed](2)--(5);
       \draw[black](7)--(3);
       \draw[black, dashed](8)--(6);
       \draw[red, dashed](1)--(3);
       \draw[red, dashed](2)--(6);
       \draw[red](7)--(4);
       \draw[red](8)--(5);
       \draw[green!75!black, dashed](1)--(6);
       \draw[green!75!black](2)--(3);
       \draw[green!75!black](7)--(5);
       \draw[green!75!black, dashed](8)--(4);
       \draw[blue](1)--(5);
       \draw[blue](2)--(4);
       \draw[blue](7)--(6);
       \draw[blue](8)--(3);
       }
 
 {\normalsize      
$\begin{array}{ccc|cccc|cc}
\qquad\ \  & \ \ \  1 & \  \ 2 & \ \ \ {\bf 3} & \ \ \ {\bf 4} & \ \ \ {\bf 5} & \ \ \  {\bf 6}\  & \ \ \ 7 & \ \  \ 8\\
\hline
Black & \ \ \ 4 & \fbox{$-5$} & \ \ \ 7 &  \ \ \ 1 & -2 & -8\  & \ \ \ 3 & \fbox{$-6$}\\
Blue & \ \ \ 5 &  \ \  4 & \ \ \ 8 & \ \ \ 2 & \ \ \ 1 & \ \ \ 7\  & \ \ \ 6 & \ \ \ 3\\
Red  & -3 & \fbox{$-6$} &  -1 & \ \ \ 7 & \ \ \ 8 & -2\  &  \ \ \ 4 & \fbox{\ \ \ $5$} \\
Green &   -6 & \ \   3 & \ \ \ 2  &  -8 & \ \ \ 7 & -1\  & \ \ \ 5 & -4 \\
\end{array}$}
\caption{Adinkra with $N=4$ and rank sequence $(2, 4, 2)$}
\end{figure}

\begin{definition}  A linear code $C$ is called {\bf doubly even} if the weight of every element in $C$ is a multiple of four.
\end{definition}

\begin{theorem} (\cite{DFG}, Theorem 4.3) Let $G$ be the pre-Adinkra associated to a connected Adinkra.  There is a one-one correspondence between such $G$ and doubly even block binary codes.  On each quotient $Q_C$ for a doubly even code $C$, there exists a totally odd dashing.
\end{theorem}

\begin{proof} Let $G$ be such a pre-Adinkra with code $C$ and base point $v_0$.  We replace each permutation $s_t$ with a signed permutation $S_t$.  Then these $S_t$ act on signed vertices; whenever we traverse a dashed edge, we multiply the vertex label by $-1.$  As before, $S_t^2=e$ and, on a bicolor 4-cycle, $S_iS_j=-S_jS_i$ when $i\ne j.$

Now suppose $x$ is in the code for $C$. Let $s_x$ be the product of all involutions $s_t$ such that the $t$-th entry in $x$ is a 1, and let $S_x$ be the product of the corresponding $S_t$, listed in some fixed order. Any permutation of the factors changes $S_x$ up to a sign.  Since $s_x(v_0)=v_0,$ then $S_x(v_0)=\pm v_0,$ and $S_x^2=e.$    If we simplify $S_x^2$, each $S_t$ cancels when squared, but signs are introduced by interchanging the order. Since there are $wt(x)$ factors of $S_t$ in $S_x$, it follows that 
$$1+ 2 + ... + (w(t)-1) =\frac{wt(x)(wt(x)-1)}2$$
must be even, or equivalently, $wt(x)$ must be a multiple of 4.

We omit the proof of existence for a totally odd dashing when $C$ is a doubly even code.  In this case, the signed permutation matrices allow for the construction of a Clifford algebra. $\square$
\end{proof}

\begin{example} The pre-Adinkra $F_6$ does not admit a totally odd dashing.  
\end{example}

If we describe a code $C$ using a basis $B=\{u_1, \dots, u_k\},$, then we represent $B$ as a $k\times N$ matrix with basis elements as rows. The doubly even property is determined from a basis by the following proposition.

\begin{proposition} Suppose $x_1$ and $x_2$ are bitstrings of length $n$. Then 
$$wt(x_1+x_2) = wt(x_1)+wt(x_2)-wt(x_1\&\, x_2),$$
where $x_1\&\, x_2$ denotes bitwise ``and".  That is, an entry in $x_1\&\, x_2$ takes value 1 only if both $x_1$ and $x_2$ take value 1 in that entry.
\end{proposition}

Therefore, if $wt(x_1)$ and $wt(x_2)$ are multiples of 4, then $wt(x_1+x_2)$ is a multiple of 4 if and only if $wt(x_1\& \, x_2)$ is even.

\begin{example} To define the $d_{2n}$ family of doubly even codes, each code has length $2n$, and  each basis element has weight $4$.  We simply stagger consecutive quartets of 1s by increments of 2 to construct the basis.  For instance, if we list basis vectors as rows,
$$d_6 =\left[ \begin{array}{cccccc}
1 & 1 & 1 & 1 & 0 & 0 \\
0 & 0 & 1 & 1 & 1 & 1
\end{array}\right],\qquad
d_8 =\left[ \begin{array}{cccccccc}
1 & 1 & 1 & 1 & 0 & 0 & 0 & 0 \\
0 & 0 & 1 & 1 & 1 & 1 & 0 & 0 \\
0 & 0 & 0 & 0 & 1 & 1 & 1 & 1 
\end{array} \right].
$$
\end{example}

Up to equivalence, doubly even codes have been classified up to $N=28$; see \cite{DFG} for tables and further interesting examples.  For supersymmetry, the main cases of interest occur when $N = 4k \le 32.$

\section{Partial Orderings}
\label{sec:8}

Once a pre-Adinkra admits a totally odd dashing, what remains in the definition for an Adinkra is the compatible partial order.  The edges determine all covering relations up to a height function on the Hasse diagram; we simply choose an orientation of the Hasse diagram that points the arrows for the partial ordering upwards.  For background on  partially ordered sets and Hasse diagrams, see, for instance, \cite{StA} or \cite{StE}.

\begin{figure}[h]
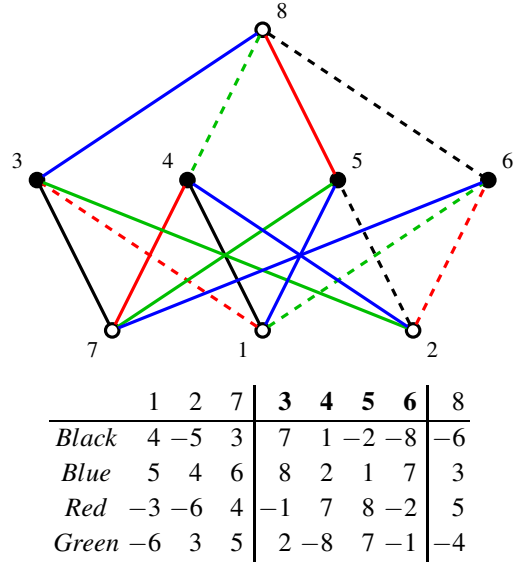

 \centering
 \TikZ{[every node/.style={inner sep=1pt, outer sep=0, shape=circle, 
                           very thick, draw=black, label distance=1mm}, very thick]
       \path[use as bounding box](-3.2,-.7)--(3.2,4.4);
       \node[fill=white, label={225:1}](1) at( 0,0){\phantom{.}};
       \node[fill=white, label={315:2}](2) at( 2,0){\phantom{.}};
       \node[fill=white, label={225:7}](7) at(-2,0){\phantom{.}};
       \node[fill=white, label={ 45:8}](8) at( 0,4){\phantom{.}};
       \node[fill=black, label={135:3}](3) at(-3,2){\phantom{.}};
       \node[fill=black, label={135:4}](4) at(-1,2){\phantom{.}};
       \node[fill=black, label={ 45:5}](5) at( 1,2){\phantom{.}};
       \node[fill=black, label={ 45:6}](6) at( 3,2){\phantom{.}};
       \draw[black](1)--(4);
       \draw[black, dashed](2)--(5);
       \draw[black](7)--(3);
       \draw[black, dashed](8)--(6);
       \draw[red, dashed](1)--(3);
       \draw[red, dashed](2)--(6);
       \draw[red](7)--(4);
       \draw[red](8)--(5);
       \draw[green!75!black, dashed](1)--(6);
       \draw[green!75!black](2)--(3);
       \draw[green!75!black](7)--(5);
       \draw[green!75!black, dashed](8)--(4);
       \draw[blue](1)--(5);
       \draw[blue](2)--(4);
       \draw[blue](7)--(6);
       \draw[blue](8)--(3);
       }
 
 {\normalsize     
 $\begin{array}{cccc|cccc|c}
 \qquad\ \  & \ \ \  1 &\ \ \ 2  & \ \ \ 7\  & \ \ \ {\bf 3} & \ \ \ {\bf 4} & \ \ \ {\bf 5} & \ \ \  {\bf 6}\  & \ \  \ 8\\
\hline
Black & \ \ \  4 &  -5 &\ \ \  3\   &  \ \ \ 7 & \ \ \  1 &  -2 &   -8\  & -6\\
Blue & \ \ \ 5 &  \ \ \ 4 & \ \ \  6\   &  \ \ \ 8 & \ \ \ 2 & \ \ \ 1 & \ \ \  7\  & \ \ \  3 \\
Red  & -3 &  -6 & \ \ \ 4\   &   -1 & \ \ \  7 & \ \ \ 8 &   -2\ & \ \ \ 5  \\
Green &  -6 & \ \ \ 3 & \ \ \ 5\ & \ \ \ 2  & -8 & \ \ \ 7 &  -1\ & -4  \\
\end{array}$      }
 \caption{Adinkra with $N=4$ and rank sequence $(3,4,1)$}
\end{figure}

\begin{figure}[h]
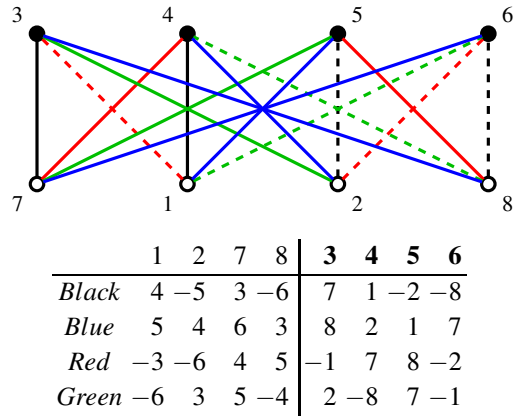

 \centering
 \TikZ{[every node/.style={inner sep=1pt, outer sep=0, shape=circle, 
                           very thick, draw=black, label distance=1mm}, very thick]
       \path[use as bounding box](-3.2,-.7)--(3.2,2.5);
       \node[fill=white, label={225:1}](1) at(-1,0){\phantom{.}};
       \node[fill=white, label={315:2}](2) at( 1,0){\phantom{.}};
       \node[fill=white, label={225:7}](7) at(-3,0){\phantom{.}};
       \node[fill=white, label={315:8}](8) at( 3,0){\phantom{.}};
       \node[fill=black, label={135:3}](3) at(-3,2){\phantom{.}};
       \node[fill=black, label={135:4}](4) at(-1,2){\phantom{.}};
       \node[fill=black, label={ 45:5}](5) at( 1,2){\phantom{.}};
       \node[fill=black, label={ 45:6}](6) at( 3,2){\phantom{.}};
       \draw[black](1)--(4);
       \draw[black, dashed](2)--(5);
       \draw[black](7)--(3);
       \draw[black, dashed](8)--(6);
       \draw[red, dashed](1)--(3);
       \draw[red, dashed](2)--(6);
       \draw[red](7)--(4);
       \draw[red](8)--(5);
       \draw[green!75!black, dashed](1)--(6);
       \draw[green!75!black](2)--(3);
       \draw[green!75!black](7)--(5);
       \draw[green!75!black, dashed](8)--(4);
       \draw[blue](1)--(5);
       \draw[blue](2)--(4);
       \draw[blue](7)--(6);
       \draw[blue](8)--(3);
       }
       
{\normalsize  $\begin{array}{ccccc|cccc}
\qquad\ \  & \ \ \  1 & \ \ \ 2 & \ \ \ 7 & \ \ \ 8\   & \ \ \ {\bf 3} & \ \ \  {\bf 4} & \ \ \ {\bf 5} & \ \ \ {\bf 6} \\
\hline
Black & \ \ \ 4 &   -5 & \ \ \ 3 & -6 \ & \ \ \  7 & \ \ \ 1 & -2 &  -8 \\
Blue  & \ \ \ 5 & \ \  \ 4 & \ \ \ 6 & \ \ \ 3\ & \ \ \ 8 & \ \ \ 2 & \ \ \ 1 & \ \ \ 7\\
Red  & -3 & -6 & \ \ \  4 & \ \ \ 5\ & -1 & \ \ \ 7 & \ \ \ 8 & -2 \\
Green & -6 & \ \ \ 3 & \ \ \ 5  &   -4\ & \ \ \ 2 & -8 & \ \ \ 7 & -1 \\
\end{array}$      }
 \caption{Adinkra with $N=4$ and rank sequence $(4,4)$}
\end{figure}

\begin{definition} The height function on a pre-Adinkra $G$ is function $h: V(G)\to \mathbb{Z}$ such that if $v$ and $w$ are adjacent in $G$, then $|h(v)-h(w)|=1.$ 
\end{definition}

With this definition, the levels in the Hasse diagram for  $G$ alternate between boson and fermion entries. We obtain a finite graded poset, possibly with several maximal and minimal elements.

For the adjacency list  $L(c)$, we indicate the the partial ordering with a {\bf lexicographic ordering} of the vertices along the column labels. If we have the Hasse diagram of $G$, we simply start at the bottom level and order the vertices from left to right as we move up the levels.  The order along each level does not matter, only the grouping by heights. 

For convenience, we label the columns for fermions in boldface, so that column label blocks alternate between regular and boldface type, and we partition the rectangle accordingly with vertical lines.  The column labels for a block occur only as entries in the columns of the adjacent blocks.  See Figs 7, 8, and 9 for examples of Adinkras as Hasse diagrams with adjacency lists.

We use similar notation for the corresponding adjacency matrix.  In applications, it is typical to indicate colors with subscripted variables, as shown in Fig 10.   With a compatible partial ordering, we subdivide the matrix into blocks according to the levels, and this matrix has nonzero entries in the blocks just above and below the main diagonal of blocks.

\begin{figure}
\centering
{\normalsize  $\left[\begin{array}{@{}rr|rrrr|rr@{}}
\  0\  & \ 0\  & -x_3 & x_1 & x_2 & -x_4 &\  0\  & \ 0\  \\ 
\  0\  & \ 0\  &  x_4 & x_2 & -x_1 & -x_3 & \ 0\  & \ 0\ \\ \hline
 -x_3 & x_4 &  \  0\  & \ 0\  & \ 0\  & \  0\  & x_1 & x_2\\
 x_1  & x_2 &   \ 0\  & \ 0\  & \ 0\  & \ 0\  & x_3 & -x_4\\
   x_2  & -x_1 &  \ 0\  &\  0\  & \ 0\  & \ 0\  & x_4 & x_3\\ 
  -x_4 & -x_3 &  \  0 \ &\  0 \  & \ 0\  & \ 0\  & x_2 & -x_1\\  \hline
  \ 0\  &\  0\  & x_1 & x_3 & x_4 & x_2 & \ 0\  & \ 0\ \\ 
 \  0\  &\  0\  &  x_2 & -x_4 & x_3 & -x_1 & \  0 \ &  \ 0\ 
\end{array}\right]$  }
\caption{Adjacency matrix for the $N=4$ and $(2, 4, 2)$ Adinkra}
\label{}
\end{figure}

Finally, another operation for creating new Adinkras from old ones is the lowering and raising of vertices in the Hasse diagram.  A vertex or grouping of vertices in a level can only move to a similar level of bosons or fermions.  Such moves must respect adjacency; for instance, if we move a single boson, it must pivot across a level of fermions.  Figs 7-9 display a sequence that moves the top two bosons to the bottom level, along with the changes to the adjacency lists.  In Fig 9, we see a canonical ordering available to every Adinkra, the {\bf valise} ordering, which places every boson at height 0 and every fermion at level 1.

\section{Motivation from Supersymmetry}
\label{sec:9}

The key concept in supersymmetry is the {\bf supermultiplet}, which consists of particles, called bosons and fermions,  and supercharges.  The particles are represented by component fields $\{b_i(t), f_i(t)\}$, which should be considered as functions of time with complex values.  The component fields are acted upon by {\bf supercharges} $Q_i$. For supermultiplets represented by Adinkras, these linear operators induce bijections on the corresponding vertices.  The corresponding action by the $Q_i$ is transitive.  Note that the conventional notations for supercharges and hypercubes coincide.

In general, each $Q_i$ commutes with $\frac{d}{dt}$, each $Q_i^2 = i \frac{d}{dt}$, and $Q_iQ_j = -Q_jQ_i$ when $i\ne j.$ The last two equations are summarized by the equations
$$Q_iQ_j+Q_jQ_i = 2 i \frac{d}{dt}.$$
For the Adinkraic case, each $Q_k$ simply interchanges the bosonic and fermionic fields, up to factors of signs, the complex number $i$, and the differential operator $\frac{d}{dt}$.  See Fig 11 for an example of a supercharge on the supermultiplet corresponding to Fig 8.

\setcounter{figure}{10}
\begin{figure}[h]
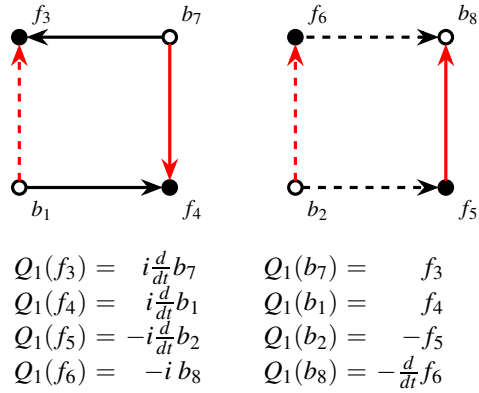

 \centering
 \TikZ{[every node/.style={inner sep=1pt, outer sep=0, shape=circle, 
                           very thick, draw=black, label distance=1mm},
        very thick, scale=1]
       \path[use as bounding box](-.2,-.8)--(2.2,2.7);
       \node[fill=white, label={315:$b_1$}](1) at(0,0){\phantom{.}};
       \node[fill=white, label={ 45:$b_7$}](7) at(2,2){\phantom{.}};
       \node[fill=black, label={ 45:$f_3$}](3) at(0,2){\phantom{.}};
       \node[fill=black, label={315:$f_4$}](4) at(2,0){\phantom{.}};
       \draw[black, -Stealth](1)--(4);
       \draw[black, -Stealth](7)--(3);
       \draw[red, dashed, -Stealth](1)--(3);
       \draw[red, -Stealth](7)--(4);
       }
\qquad\qquad
 \TikZ{[every node/.style={inner sep=1pt, outer sep=0, shape=circle, 
                           very thick, draw=black, label distance=1mm},
        very thick, scale=1]
       \path[use as bounding box](-.2,-.8)--(2.2,2.7);
       \node[fill=white, label={315:$b_2$}](2) at(0,0){\phantom{.}};
       \node[fill=white, label={ 45:$b_8$}](8) at(2,2){\phantom{.}};
       \node[fill=black, label={ 45:$f_6$}](6) at(0,2){\phantom{.}};
       \node[fill=black, label={315:$f_5$}](5) at(2,0){\phantom{.}};
       \draw[black, dashed, -Stealth](2)--(5);
       \draw[black, dashed, Stealth-](8)--(6);
       \draw[red, dashed, -Stealth](2)--(6);
       \draw[red, Stealth-](8)--(5);
       }       
 
{\normalsize       
$\begin{array}{ccc}
Q_1(f_3) & = & \ \ i \frac{d}{dt} b_7\\
Q_1(f_4) & = & \ \ \ i \frac{d}{dt}  b_1\\
Q_1(f_5) & = & -i \frac{d}{dt} b_2 \\
Q_1(f_6) & = & \ \ \  -i \ b_8
\end{array}\qquad
\begin{array}{ccc}
Q_1(b_7) & = & \quad\ \ \  f_3\\
Q_1(b_1) & = &  \quad\ \ \ f_4\\
Q_1(b_2) & = &  \quad  - f_5 \\
Q_1(b_8) & = & -\frac{d}{dt} f_6
\end{array}$    }
 \caption{Black-red 4-cycles and equations for $Q_1$ in Fig~8}
\end{figure}

Such equations can be represented graphically as follows:
\begin{itemize}[noitemsep]
\item each $b_i$ and $f_i$ contribute a vertex, as an open circle or closed circle, respectively,
\item an edge of color $k$ connects $b_i$ and $f_j$ if $Q_k$ interchanges $b_i$ and $f_j$, and 
\item these edges are decorated in accordance with the factors as noted below.
\end{itemize}

We reconstruct part of the supermultiplet from Fig 8.  Suppose $Q_1$ represents the supercharge along the black edges. If each edge is treated as an arrow pointing up, then we reconstruct $Q_1$ on fermions as follows:

\begin{itemize}[noitemsep]
\item by convention, the images of $Q_1$ on a fermion carries a factor of the scalar $i$,  
\item if a boson points to a fermion with a black edge, then the image of $Q_1$ on the fermion has a derivative factor, and
\item if the edge is dashed, the image of the fermion carries a negative sign.
\end{itemize}

Since each $Q_1^2=i\frac{d}{dt}$, the rules for images of bosons follow immediately.  The rules for $Q_1$ in Fig 8 are given in Fig 11 and  use three of the four types of arrows.  Furthermore, if we record the rules for the red edges as $Q_2$, then the equation $Q_iQ_j=-Q_jQ_i$ is represented by the totally odd dashing on the bicolor 4-cycles in Fig 11.

\section*{Acknowledgements}
This material is based upon work supported by the National Science
Foundation (NSF) under Grant No. DMS-1928930, National Security Agency
(NSA) under Grant No. H98230-24-1-0020, and the
Sloan Foundation under Grant G-2020-14104 while the authors participated
in ADJOINT 2024, hosted by the Mathematical Sciences Research Institute in
Berkeley, California.   We thank Edray Goins, the on-site director of program, which ran from June 24 to July 5, 2024.   

We also thank Charles Doran and Stefan M{\'e}ndez-Diez for organizing a one-week workshop ``MathScape 2024: The Mathematics of Supersymmetry'' at Bard College, Aug. 12--16, 2024.   

The research of S.J.G. is currently supported by the Clark Leadership Chair in Science endowment at the University of Maryland - College Park. T.H. is grateful to the Department of Physics, University of Maryland, and the Physics Department of the University of Novi Sad, Serbia, for recurring hospitality and resources.  R.N. would like to acknowledge support from National Science Foundation Award 2150252 (REU Site: Queens Experiences in Discrete Mathematics) and the Office of Academic Affairs, York College, City University of New York.
%
%
%
%

\end{document}